\documentclass[a4paper, twoside, 12pt, final]{amsart}
\usepackage[OT1]{fontenc}
\usepackage[latin1]{inputenc}
\usepackage{amssymb,latexsym}
\usepackage{type1cm,ae}
\usepackage[notref,notcite]{showkeys}
\usepackage{graphicx}
\usepackage{psfrag}

\theoremstyle{plain}		\newtheorem{theorem}{Theorem}

				\newtheorem{lemma}{Lemma}
\theoremstyle{definition}	
				\newtheorem{example}{Example}
\theoremstyle{remark}		

\newcommand*{\bR}{\ensuremath{\mathbb{R}}}
\newcommand*{\bN}{\ensuremath{\mathbb{N}}}

\newcommand*{\loc}{\mathrm{loc}}
\newcommand*{\closure}[1]{\overline{#1}}

\newcommand*{\Wert}{\mathord{\mbox{|\kern-1.5pt|\kern-1.5pt|}}}

\DeclareMathOperator{\dist}{dist}
\DeclareMathOperator{\Det}{Det}

\DeclareMathOperator{\diam}{diam}

\def\XXint#1#2#3{{\setbox0=\hbox{$#1{#2#3}{\int}$}
     \vcenter{\hbox{$#2#3$}}\kern-.5\wd0}}

\title[Generalized dimension distortion]{Generalized dimension distortion under mappings of sub-exponentially integrable distortion}

\author[ T. Rajala ]{ Tapio Rajala }
\address[ T. Rajala ]{Department of Mathematics and Statistics, University of Jyv\"askyl\"a, P.O. Box 35, Fin-40014 University of Jyv\"askyl\"a, Finland}
\email{tapio.m.rajala@jyu.fi}

\author[A. Zapadinskaya]{Aleksandra Zapadinskaya}
\address[A. Zapadinskaya]{Department of Mathematics and Statistics, University of Jyv\"askyl\"a, P.O. Box 35, Fin-40014 University of Jyväskyl\"a, Finland}
\email{aleksandra.zapadinskaya@jyu.fi}

\author[T. Z\"urcher]{Thomas Z\"urcher}
\address[T. Z\"urcher]{Department of Mathematics and Statistics, University of Jyv\"askyl\"a, P.O. Box 35, Fin-40014 University of Jyväskyl\"a, Finland}
\email{thomas.t.zurcher@jyu.fi}


\thanks{The second author was partially supported by the Academy of Finland, grant no.~120972, and the third author was supported by the Swiss National Science Foundation.}

\begin{document}
\begin{abstract}
We prove a dimension distortion estimate for mappings of sub-exponentially integrable distortion in Euclidean spaces, which is essentially sharp in the plane.
\end{abstract}
\maketitle
\section{Introduction}

The roots of our studies lie in~\cite{Geh-Väi}, where the following was proved:
given a planar $K$\nobreakdash-quasiconformal mapping~$f$ and a set
$E$ with $\dim_{\mathcal H} E<2$, we have $\dim_{\mathcal H} f(E)\leq \beta<2$, where
$\beta$ depends only on $K$ and the Hausdorff dimension $\dim_{\mathcal H} E$ of the set $E$. 
Later, it was shown that the same is true in higher dimensions with $\beta$ depending on the dimension of the underlying space as well as on $K$ and on $\dim_{\mathcal H} E$ (see~\cite{gehring})
. These results rely on the higher integrability of the Jacobian of a quasiconformal mapping~\cite{BojHigher,gehring}.

Recent extensions take a wider class of mappings into consideration. A continuous mapping $f\in W^{1,1}_\loc(\Omega,\bR^n)$ ($\Omega\subset\bR^n$ is a domain) is called a \emph{mapping of finite distortion}, if its Jacobian is locally integrable and there exists a measurable function $K\colon\Omega\to[1,\infty[$ such that
$$
|Df(x)|^n\leq K(x)J_f(x)
$$
for almost every $x\in\Omega$. An assumption on $K$ that still guarantees a lot of the
properties of quasiconformal mappings is the so-called
exponential integrability. This condition requires
that $\exp(\lambda K)$ is locally integrable for some $\lambda$. In this case, $f$ is called a \emph{mapping of $\lambda$-exponentially integrable distortion}.

Such mappings satisfy Lusin's condition~N, i.e. they map sets of measure zero to sets of measure zero, \cite{KKM}. However, in \cite[Proposition~5.1]{HK03}, a mapping $f\colon\bR^n\to\bR^n$ of finite exponentially integrable distortion that maps sets of Hausdorff dimension less than $n$ to sets of Hausdorff dimension $n$ was constructed.

Still it was possible to obtain reasonable dimension distortion results in terms of generalized Hausdorff measure (see the next section for the definition). In \cite{HK03}, it was shown that there exists a constant $k_n$, depending only on $n$, such that if $f\colon \bR^n\to \bR^n$ is a homeomorphism with $\lambda$-exponentially integrable distortion for some $\lambda$, then $\mathcal{H}^h(f(S^{n-1}))<\infty$ for all $p<k_n\lambda$, where $\mathcal{H}^h$ is the generalized Hausdorff measure with gauge function $h(t)=t^n\log^p(1/t)$.

A sharp result of this kind in the planar case was obtained in~\cite{KZZ}, where the circle $S^1$ was replaced by a general set $E$ of Hausdorff dimension less than two: we have $\mathcal{H}^h(f(E))=0$ for all $p<\lambda$, where $h(t)=t^2\log^p(1/t)$, if $f$ is a mapping of $\lambda$-exponentially integrable distortion. The proof is based on the higher regularity for the weak derivatives of the mapping $f$~\cite{eero} and dimension distortion estimates for Orlicz-Sobolev mappings. See~\cite{KZZ1,tapio} for related results in the plane and~\cite{ndim} for the generalization to higher dimensions.

The assumption of exponential integrability for the distortion is further relaxed by replacing it with a more general Orlicz condition. That is, one may assume that
$$
e^{\mathcal A(K)}\in L^1_\loc,\,\,\,\,\,\,\,\text{where}\,\,\,\,\,\,\int\limits_1^\infty\frac{\mathcal A(t)}{t^2}dt=\infty,
$$
for a distortion function $K$ of a mapping of finite distortion $f$\linebreak(see~\cite[Section~20.5]{book}). In particular, when $\mathcal A(t)=p\frac{t}{1+\log t}-p$, for some $p>0$, such mapping $f$ is called a \emph{mapping of sub-exponentially integrable distortion}. Dimension distortion in this particular case is examined in this paper.

Let us agree that from now on, $\Omega$ is always an open set in $\bR^n$, $n\geq2$. Denote $h_{n,\beta}(t)=t^n(\log\log(1/t))^\beta$. We have the following theorem.
\begin{theorem}\label{ndim}
There exists a constant $c>0$, which depends only on the dimension $n$ of the underlying space, such that for every homeomorphism of finite distortion $f\in W^{1,1}_\loc(\Omega;\bR^n)$, $\Omega\subset\bR^n$, with
$$
e^{\frac{K_f}{1+\log K_f}}\in L^p_\loc(\Omega),
$$
we have $\mathcal H^{h_{n,\beta}}(f(E))=0$ for all $\beta<cp$, whenever $E\subset\Omega$ is such that $\dim_{\mathcal H}E<n$.
\end{theorem}
When $n=2$, the assumption on $f$ to be a homeomorphism is not necessary due to Stoilow factorization. The constant $c$ equals one in this case:
\begin{theorem}\label{planar}
Let $f\in W^{1,1}_\loc(\Omega;\bR^2)$, $\Omega\subset\bR^2$, be a mapping of finite distortion with
$$
e^{\frac{K_f}{1+\log K_f}}\in L^p_\loc(\Omega).
$$
Then $\mathcal H^{h_{2,\beta}}(f(E))=0$ for all $\beta<p$, whenever $E\subset\Omega$ is such that $\dim_{\mathcal H}E<2$.
\end{theorem}
The following example shows that Theorem~\ref{planar} is essentially sharp:
\begin{example}\label{ex}
For any $\beta>0$ and $\varepsilon\in]0,\beta[$, there exist sets $\mathcal C,\,\,\mathcal C^\prime\subset[0,1]^2$, such that $\dim_{\mathcal H}\mathcal C<2$ and $\mathcal H^{h_{2,\beta}}(\mathcal C^\prime)>0$, and a mapping $f\in W^{1,1}([0,1]^2;\bR^2)$, such that
$$
e^{\frac{K_f}{1+\log K_f}}\in L^{\beta-\varepsilon}_\loc(\Omega)
$$
and $f(\mathcal C)=\mathcal C^\prime$.
\end{example}
This example can be extended to higher dimensions. In this case, the gauge function for the image set $\mathcal C^\prime$ is $h_{n,\beta}$ and the distortion of $f$ satisfies the same sub-exponential integrability condition. Thus, one may expect that the sharp value of the constant $c$ in Theorem~\ref{ndim} is one as well.

The main auxillary result, used in the proof of the theorems, is higher integrability for the Jacobian of a mapping of sub-exponentially integrable distortion, proved in~\cite{clop} for general dimensions and refined in~\cite{gill}, where a sharp estimate for the higher integrability of the Jacobian of a planar mapping was obtained. Those estimates are combined with the methods used in~\cite{KZZ1, tapio} for the case of exponentially integrable distortion.

One could extend the results presented here to a case of a more general function $\mathcal A$, in particular, when $\mathcal A$ is given by
\begin{multline*}
\mathcal A_{p,k}(t)\\=\frac{pt}{1+\log(t)\log(\log(e-1+t))\cdots\log(\ldots(\log(e^{e^{\cdot^{\cdot^{\cdot^e}}}}-1+t))\ldots)}-p,
\end{multline*}
where $k$ means that the last logarithmic expression is a $k$-th iterated logarithm (a case, studied in~\cite[Theorem~4]{gill}). However, we leave the results in the presented form, because the construction demonstrating sharpness is quite complicated even in the case of a single logarithm.

\section{Definitions}

Let us agree on some notation. For a set $V\subset\bR^n$ and a number $\delta>0$, $V+\delta$ denotes
the set $\{y\in\bR^n\,|\dist(y,V)<\delta\}$.

Always when we introduce a constant using the notation $C = C(\cdot)$, we mean that the
constant $C$ depends only on the parameters listed in the parantheses.

We write $\mathcal{H}^h(A)$ for the \emph{generalized Hausdorff
measure} of a set $A$, given by
$$
\mathcal{H}^h(A)=\lim_{\delta\to0}\mathcal{H}^h_\delta(A),
$$
where
$$
\mathcal{H}^h_\delta(A)=\inf\Bigl\{\sum\limits_{i=1}^{\infty}
h(\diam U_i)\colon A\subset\bigcup\limits_{i=1}^{\infty} U_i,
\diam U_i\leq\delta\Bigr\}
$$
and $h$ is a dimension gauge (non-decreasing, $\lim_{t\to0+}h(t)=h(0)=0$). If
$h(t)=t^\alpha$ for some $\alpha\geq0$, we simply put
$\mathcal{H}^\alpha$ for $\mathcal{H}^{t^\alpha}$ and call it the
\emph{Hausdorff $\alpha$-dimensional measure} and the
\emph{Hausdorff dimension} $\dim_{\mathcal{H}}A$ of the set $A$
is the smallest $\alpha_0\geq0$ such that
$\mathcal{H}^\alpha(A)=0$ for any $\alpha>\alpha_0$.

Let us recall the definition of Orlicz classes. An \emph{Orlicz function} is a continuous increasing function $P\colon[0,\infty[\to[0,\infty[$ such that $P(0)=0$ and $\lim_{t\to\infty}P(t)=\infty$. Given an Orlicz function $P$, we denote by $L^P(\Omega)$ the \emph{Orlicz class} of integrable functions $h:\Omega\to\bR $ such that
$$
\int_\Omega P(\nu|h|)<\infty
$$
for some $\nu=\nu(f)>0$. An \emph{Orlicz-Sobolev class} $W^{1,P}(\Omega)$ is a class of mappings $g\in W^{1,1}(\Omega,\bR^2)$ such that all the partial derivatives of $g$ are in the class $L^P(\Omega)$.

Finally, given a mapping $f\in W_\loc^{1,1}(\Omega, \bR^n)$,  we write the equality $\Det Df=J_f$, if the distributional determinant $\Det Df$~\cite{ball} coincides with the pointwise Jacobian $J_f$, that is, if
$$
\int_\Omega f_1(x)J_{\tilde{f}}(x)dx=-\int_\Omega\varphi(x)J_f(x)dx
$$
holds for each $\varphi\in C_0^\infty(\Omega)$ (here $f=(f_1,\ldots,f_n)$ and $\tilde{f}=(\varphi,f_2,\ldots,f_n)$). See~\cite{iwasbo,greco,greco2,koszho} for some conditions on the regularity of the weak derivatives of $f$ sufficient to guarantee this equality.

\section{Example}

Fix $\beta>0$. Let us construct the mapping in Example~\ref{ex}. We start by defining the pre-image and image Cantor sets $\mathcal C$ and $\mathcal C^\prime$, respectively. Fix $\sigma\in]0,1/2[$. The set $\mathcal C$ is obtained as a Cartesian product $\mathcal C_1\times\mathcal C_1$, where $\mathcal C_1$ is a Cantor set on the real line. In order to construct $\mathcal C_1$, take a unit segment $I=[0,1]$ and divide it into eight equal parts. Consider eight intervals $I^3_j$, $j=1,\ldots,8$, of length $\sigma^3$, each taken in the middle of one of the obtained segments. At the further steps, the intervals considered are always divided into two parts. Given $2^{k}$, $k\geq3$, intervals $I^k_j$, $j=1,\ldots,2^{k}$, of length $\sigma^k$, we divide each of them into two parts and take $2^{k+1}$ intervals $I^{k+1}_j$, $j=1,\ldots,2^{k+1}$, of length $\sigma^{k+1}$, each in the middle of one of the obtained parts. Finally, $\mathcal C_1$ is taken as $\bigcap\limits_{k\geq3}\bigcup\limits_{j=1}^{2^k}I^k_j$. The Hausdorff measure $\mathcal H^\alpha(\mathcal C_1)$ of the set $\mathcal C_1$ for $\alpha\in]\frac{\log2}{\log(1/\sigma)},1[$ may be estimated as
$$
\mathcal H^\alpha(\mathcal C_1)\leq\inf_{k\geq3}\{2^k\sigma^{\alpha k}\}=0,
$$
so, $\dim_{\mathcal H}\mathcal C_1<1$, and thus, $\dim_{\mathcal H}(\mathcal C_1\times\mathcal C_1)<2$.

The image set $\mathcal C^\prime$ is constructed similarly, but at the $k$-th step, $k\geq3$, the length of the intervals chosen is $l_k=2^{-k}\log^{-\beta/2}k$ instead of $\sigma^k$. For any $k\geq3$, the set $\mathcal C^\prime$ can be covered by $2^{2k}$ squares of side length $l_k$. We have 
$$
\lim_{k\to\infty}2^{2k}h_{2,\beta}(l_k)=\lim_{k\to\infty}2^{2k}l^2_k(\log\log(1/l_k))^\beta=1,
$$
so the mass distribution principle gives us $\mathcal H^{h_{2,\beta}}(\mathcal C^\prime)>0$; indeed, put $m:=\inf_{k\geq3}\{2^{2k}h_{2,\beta}(l_k)\}>0$ and let $\mu$ be the uniformly distributed probability measure supported by $\mathcal C^\prime$. Suppose also that $\delta>0$ is so small that $h_{2,\beta}(t)$ is increasing in $t$ on the interval $]0,\delta[$. Then for any $U\subset\bR^2$ such that $l_{k+1}\leq\diam U<\min\{\delta,l_k\}$ for some $k\geq3$, we have
$$
\mu(U)\leq2^{-2k}\leq\frac{4h_{2,\beta}(l_{k+1})}{m}\leq\frac{4h_{2,\beta}(\diam U)}{m}.
$$
Thus, for any covering $\bigcup_i U_i$ of the set $\mathcal C^\prime$, such that $\diam U_i<\min\{\delta,l_3\}$, $i=1,2,\ldots$, we observe
$$
\sum\limits_{i=1}^\infty h_{2,\beta}(\diam U_i)\geq\frac{m}{4}\sum\limits_{i=1}^\infty\mu(U_i)
\geq\frac{m}{4}\mu\Bigl(\bigcup\limits_{i=1}^\infty U_i\Bigr)=\frac{m}{4}>0.
$$

Let us denote by $Q_{k,j}$ with $k=3,4,\ldots$ and $j=1,\ldots,2^{2k}$ the squares of the side length $\sigma^k$, appearing on the pre-image side at the $k$-th step of the construction. Write $q_{k,j}$ for the centres of these squares. Next, let $A_{k,j}$ for $k=3,4,\ldots$ and $j=1,\ldots,2^{2k}$ denote the frames
$$
\{x\in\bR^2\colon r_k<|x-q_{k,j}|_\infty<R_k\},
$$
where $r_k=\sigma^k/2$ for $k\geq3$, $R_k=\sigma^{k-1}/4$ for $k\geq4$, $R_3=1/16$ and $|\cdot|_\infty$ is the maximum norm:
$$
|x|_\infty=\max\{|x_1|,|x_2|\}.
$$ 
The inner boundary $\{x\in\bR^2\colon|x-q_{k,j}|_\infty=r_k\}$ of the frame $A_{k,j}$ is exactly the boundary of the square $Q_{k,j}$. Let us introduce similar notation for the image side. Write $Q^\prime_{k,j}$ with $k=3,4,\ldots$ and $j=1,\ldots,2^{2k}$ for the squares with the side length $l_k=2^{-k}\log^{-\beta/2}k$ and $q^\prime_{k,j}$ for the centres of these squares. Finally, $A^\prime_{k,j}$ for $k=3,4,\ldots$ and $j=1,\ldots,2^{2k}$ denote the frames
$$
\{x\in\bR^2\colon r^\prime_k<|x-q^\prime_{k,j}|_\infty<R^\prime_k\},
$$
where $r_k^\prime=2^{-k+1}\log^{-\beta/2}k$ for $k\geq3$, $R_k^\prime=2^{-k+1}\log^{-\beta/2}(k-1)$ for $k\geq4$ and $R_3^\prime=1/16$.

We are ready to construct a mapping $f\colon[0,1]^2\to\bR^2$ such that $f(\mathcal C)=\mathcal C^\prime$. The construction is similar to the one in~\cite[Proposition~5.1]{HK03}. First, let
$$
a_k=\frac{R_k^\prime-r_k^\prime}{R_k-r_k}\,\,\,\,\text{ and }\,\,\,\,
b_k=\frac{R_kr_k^\prime-R_k^\prime r_k}{R_k-r_k},
$$
for $k\geq3$. Then, define $f_3$ as
$$
f_3(x)=
\begin{cases}
(a_3|x-q_{3,j}|_\infty+b_3)\frac{x-q_{3,j}}{|x-q_{3,j}|_\infty}+q^\prime_{3,j},& x\in \closure A_{3,j}, j=1,\ldots,64,\\
\frac{r_3^\prime}{r_3}(x-q_{3,j})+q^\prime_{3,j},& x\in Q_{3,j}, j=1,\ldots,64.
\end{cases}
$$
We proceed by putting
$$
f_k(x)=
\begin{cases}
(a_k|x-q_{k,j}|_\infty+b_k)\frac{x-q_{k,j}}{|x-q_{k,j}|_\infty}+q^\prime_{k,j},& x\in A_{k,j}, j=1,\ldots,2^{2k},\\
\frac{r_k^\prime}{r_k}(x-q_{k,j})+q^\prime_{k,j},& x\in \closure Q_{k,j}, j=1,\ldots,2^{2k},\\
f_{k-1}(x),& \text{otherwise},
\end{cases}
$$
for $k>3$. The mapping $f$ is obtained as a pointwise limit \linebreak$f=\lim_{k\to\infty}f_k$.

It is a Sobolev mapping. Indeed, let us first see that it is ACL (absolutely continuous on lines). Take a line on the pre-image side parallel to the $x_1$-axis that does not hit the initial Cantor set $\mathcal C$. On this line, the mapping $f$ coincides with one of the mappings $f_{k_0}$ in our sequence, which is Lipschitz and, therefore, absolutely continuous along the considered line. Since $\mathcal C_1$ has vanishing Lebesgue measure $\mathcal L^1$, it follows that $f$ is ACL. Next, let us check the integrability of the differential of $f$. Its behaviour is essentially defined by the behaviour of $f$ on cubical collars $A_{k,j}$, where it is given by $$(a_k|x|_\infty+b_k)\frac{x}{|x|_\infty},\,\,\,\text{    }r_k<|x|_\infty<R_k$$
up to a translation. Further calculations show that
$$
|Df(x)|=|Df_k(x)|=\max\Bigl\{a_k,a_k+\frac{b_k}{|x-q_{k,j}|_\infty}\Bigr\}\,\,\,\text{ for a. e. }x\in A_{k,j}.
$$
Since $b_k>0$ for large $k$, we have $|Df(x)|=a_k+\frac{b_k}{|x-q_{k,j}|_\infty}\leq r^\prime_k/r_k$ for almost every $x\in A_{k,j}$, when $k$ is large enough. So, the integrability of the differential of $f$ can be estimated with help of the following series:
$$
\int_{[0,1]^2}|Df|\leq C_1\sum\limits_{k=3}^\infty (2\sigma)^{2(k-1)}\frac{2^{-k+2}\log^{-\beta/2}k}{\sigma^k}=
C_2\sum\limits_{k=3}^\infty (2\sigma)^k\log^{-\beta/2}k,
$$
where $C_1=C_1(\sigma,\beta)$ and $C_2=C_2(\sigma,\beta)$ are positive constants. This series converges by the Ratio Test, since
$$
\lim_{k\to\infty}\frac{\log^{-\beta/2}(k+1)}{\log^{-\beta/2}k}=1<\frac{1}{2\sigma}.
$$
So, we have $Df\in L^1$ and therefore $f\in W^{1,1}$.

The Jacobian of $f$ is integrable as a Jacobian of a Sobolev homeomorphism (see, for example,~\cite[Corollary~3.3.6]{book}).

Finally, let us examine the sub-exponential integrability of the distortion function of $f$. The Jacobian of $f$ is given by
$$
J_{f_k}(x)=a_k\Bigl(a_k+\frac{b_k}{|x-q_{k,j}|_\infty}\Bigr)
$$
at almost every $x\in A_{k,j}$. Thus, $K_f$ is defined by
\begin{equation}\label{Kk}
K_{f_k}(x)=1+\frac{b_k}{a_k|x-q_{k,j}|_\infty}\leq\frac{1-2\sigma}{2\sigma}\frac{1}{\bigl(\frac{\log k}{\log(k-1)}\bigr)^{\beta/2}-1}=\colon K_k
\end{equation}
for almost every $x\in A_{k,j}$, when $k$ is large enough. This gives the estimate
$$
\int_{[0,1]^2}\exp\Bigl(\frac{pK_f}{1+\log K_f}\Bigr)\leq C\sum\limits_{k=3}^\infty(2\sigma)^{2(k-1)}\exp\Bigl(\frac{pK_{k}}{1+\log K_{k}}\Bigr)
$$
with a constant $C=C(\sigma,\beta)>0$. By Lemma~\ref{lem:limit} below,
\begin{equation}\label{eq:limit}
\lim_{k\to\infty}\frac{\exp\bigl(\frac{pK_{k+1}}{1+\log K_{k+1}}\bigr)}{\exp\bigl(\frac{pK_{k}}{1+\log K_{k}}\bigr)}
=\exp\Bigl(p\frac{1-2\sigma}{2\sigma}\frac{2}{\beta}\Bigr),
\end{equation}
and thus, by the Ratio Test, the series above converges provided
$$
\exp\Bigl(p\frac{1-2\sigma}{2\sigma}\frac{2}{\beta}\Bigr)<(2\sigma)^{-2}.
$$
So, we have
$$
e^{\frac{K_f}{1+\log K_f}}\in L^p_\loc(\Omega)
$$
for all $p<p_0=\beta\frac{2\sigma}{1-2\sigma}\log\frac{1}{2\sigma}$. Choosing $\sigma$ close enough to 1/2, we can make $p_0$ as close to $\beta$ as we wish.

The following lemma verifies~\eqref{eq:limit}.
\begin{lemma}\label{lem:limit}
We have
$$
\lim_{k\to\infty}\frac{\exp\bigl(\frac{pK_{k+1}}{1+\log K_{k+1}}\bigr)}{\exp\bigl(\frac{pK_{k}}{1+\log K_{k}}\bigr)}
=\exp\Bigl(p\frac{1-2\sigma}{2\sigma}\frac{2}{\beta}\Bigr),
$$
where $K_k$ is as defined in~\eqref{Kk}.
\end{lemma}
\begin{proof}
Straightforward calculations give us
\begin{multline*}
\frac{pK_{k+1}}{1+\log K_{k+1}}-\frac{pK_{k}}{1+\log K_{k}}
\\=p\alpha\frac{\bigl(\frac{1}{T_{k+1}}-\frac{1}{T_{k}}\bigr)\log^{-1}\frac{\alpha}{T_{k+1}}
\log^{-1}\frac{\alpha}{T_{k}}+\frac{1}{T_{k+1}}\log^{-1}\frac{\alpha}{T_{k+1}}-
\frac{1}{T_{k}}\log^{-1}\frac{\alpha}{T_{k}}}{1+\log^{-1}\frac{\alpha}{T_{k+1}}\log^{-1}\frac{\alpha}{T_{k}}
+\log^{-1}\frac{\alpha}{T_{k+1}}+\log^{-1}\frac{\alpha}{T_{k}}},
\end{multline*}
where $\alpha=(1-2\sigma)/(2\sigma)$ and $T_t=(\log t/\log(t-1))^{\beta/2}-1$ for $t\in[3,\infty[$. Notice that $T_t\to0$ as $t\to\infty$. Thus, in order to prove this lemma, it is enough to show that the numerator of the fraction above goes to $2/\beta$ as $k$ tends to infinity. We demonstrate it by the following two observations:
$$
\lim_{k\to\infty}\Bigl(\frac{1}{T_{k+1}}-\frac{1}{T_{k}}\Bigr)\log^{-1}\frac{\alpha}{T_{k+1}}
\log^{-1}\frac{\alpha}{T_{k}}=0
$$
and
$$
\lim_{k\to\infty}\Bigl(\frac{1}{T_{k+1}}\log^{-1}\frac{\alpha}{T_{k+1}}-
\frac{1}{T_{k}}\log^{-1}\frac{\alpha}{T_{k}}\Bigr)=\frac{2}{\beta}.
$$
The main tool here is the mean-value theorem. Let us first examine the difference $\frac{1}{T_{k+1}}-\frac{1}{T_{k}}$. There exists a sequence $\{\zeta_k\}_{k=3}^\infty$ of numbers between 0 and 1 such that
$$
\frac{1}{T_{k+1}}-\frac{1}{T_{k}}=u(k+1)-u(k)=u^\prime(k+\zeta_k),
$$
where
$$
u(t)=\frac{\log^{\beta/2}(t-1)}{\log^{\beta/2}t-\log^{\beta/2}(t-1)}.
$$
We have
$$
u^\prime(t)=\frac{\beta}{2}\frac{\log^{\beta/2}(t-1)\log^{\beta/2}t(\frac{1}{t-1}\log^{-1}(t-1)
-\frac{1}{t}\log^{-1}t)}{(\log^{\beta/2}t-\log^{\beta/2}(t-1))^2}.
$$
We apply the mean-value theorem again in order to replace the differences both in the numerator and in the denominator with multiplicative terms. We obtain for $t>3$
\begin{align*}
u^\prime(t)&=\frac{2}{\beta}\frac{(t-\theta_t)^2}{(t-\eta_t)^2}\frac{\log^{\beta/2}(t-1)\log^{\beta/2}t
(\log(t-\eta_t)+1)}{\log^{\beta-2}(t-\theta_t)\log^2(t-\eta_t)}\\
&<\frac{2}{\beta}\frac{t^2}{(t-1)^2}\frac{\log^{\beta+2} t(\log t+1)}{\log^{\beta+2}(t-1)}
<\frac{18\cdot2^\beta}{\beta}(\log t+1),
\end{align*}
where $\eta_t,\theta_t\in]0,1[$.

Next, let us observe that
\begin{align}\label{1overT}
\frac{1}{T_t}=\frac{\log^{\beta/2}(t-1)}{\log^{\beta/2}t-\log^{\beta/2}(t-1)}&=\frac{2}{\beta}
\frac{(t-\delta_t)\log^{\beta/2}(t-1)}{\log^{\beta/2-1}(t-\delta_t)}\\\nonumber&=\frac{2}{\beta}(t-\delta_t)
M_t\log(k-\delta_k),
\end{align}
where $\delta_t\in]0,1[$ and $M_t=(\log(t-1)/\log(t-\delta_t))^{\beta/2}\to1$ as $t\to\infty$. Finally, we obtain for large $k$
\begin{align*}
0&<\Bigl(\frac{1}{T_{k+1}}-\frac{1}{T_{k}}\Bigr)\log^{-1}\frac{\alpha}{T_{k+1}}
\log^{-1}\frac{\alpha}{T_{k}}\\
&<\frac{18\cdot2^\beta}{\beta}\frac{\log(k+1)+1}{(\log(k-1)+\log(\frac{2\alpha}{\beta}M_k\log(k-1)))(\log k+\log(\frac{2\alpha}{\beta}M_{k+1}\log k))}\\
&<\frac{18\cdot2^\beta}{\beta}\frac{\log(k+1)+1}{\log^2(k-1)}\to0
\end{align*}
as $k\to\infty$.

It remains to examine the difference
$$
\frac{1}{T_{k+1}}\log^{-1}\frac{\alpha}{T_{k+1}}-\frac{1}{T_{k}}\log^{-1}\frac{\alpha}{T_{k}}
=v(k+1)-v(k),
$$
where $v(t)=\frac{1}{T_t}\log^{-1}\frac{\alpha}{T_t}$. Obviously, it is enough to prove that\linebreak $\lim_{t\to\infty}v^\prime(t)=2/\beta$. Let us calculate
\begin{align*}
v^\prime(t)&=\frac{\beta}{2}\frac{\log^{\beta/2-1}t}{\log^{\beta/2+1}(t-1)}\frac{t\log t-(t-1)\log(t-1)}{t(t-1)}
\frac{1-\log^{-1}\frac{\alpha}{T_t}}{T_t^2\log\frac{\alpha}{T_t}}\\
&=\frac{\beta}{2}\frac{\log^{\beta/2-1}t(\log(t-\kappa_t)+1)}{t(t-1)\log^{\beta/2+1}(t-1)}
\frac{1-\log^{-1}\frac{\alpha}{T_t}}{T_t^2\log\frac{\alpha}{T_t}}\\
&=\frac{\beta}{2}N_t\frac{1-\log^{-1}\frac{\alpha}{T_t}}{t(t-1)T_t^2\log(t-1)\log\frac{\alpha}{T_t}},
\end{align*}
where $\kappa_t\in]0,1[$ and $N_t\to1$ as $t\to\infty$. We use the representation~\eqref{1overT} again to obtain
$$
v^\prime(t)=\frac{2}{\beta}\frac{(t-\delta_t)^2}{t(t-1)}
\frac{N_tM_t^2\log^2(t-\delta_t)(1-\log^{-1}\frac{\alpha}{T_t})}
{\log(t-1)(\log(t-\delta_t)+\log(\frac{2\alpha}{\beta}M_t\log(t-\delta_t)))}\to\frac{2}{\beta}
$$
as $t\to\infty$.
\end{proof}

\section{Proof of Theorem~\ref{ndim}}

Without loss of generality, we may assume for the rest of the paper that $\Omega$ is connected. Moreover, using the $\sigma$\nobreakdash-additivity of the generalized Hausdorff measure, we may assume in what follows, that $\Omega$ is bounded and $e^{\frac{pK_f}{1+\log K_f}}$ is globally integrable in $\Omega$. We will use a higher integrability result for the Jacobian from~\cite{clop} to establish the desired dimension distortion estimate.
\begin{proof}[Proof of Theorem~\ref{ndim}]
Corollary~3.3 from~\cite{clop} gives us a constant $c=c(n)>0$ such that $|Df|\in L_\loc^{P_\beta}(\Omega)$ for all $\beta<cp$, where 
$$
P_\beta(t)=\frac{t^n}{\log(e+t)\log^{1-\beta}(\log(e^e+t))}.
$$
Corollary~9.1 from~\cite{ivaver} implies in turn that $J_f\log^\beta\log(e^e+J_f)\in L^1_\loc(\Omega)$ for all $\beta<cp$. Fix some $q\in]n-1,n[$. The integrability of the differential of $f$ guarantees that $f\in W^{1,q}(\Omega)$. In order to conclude $f^{-1}\in W^{1,q}_\loc(f(\Omega))$ by~\cite[Theorem~4.2]{HeKoMa}, we also need $K_f^{\frac{(q-1)q}{2q-n}}$
to be integrable in $\Omega$, which is clearly true as $K_f$ is sub-exponentially integrable. Finally, the regularity of the weak derivatives of $f$ is enough to guarantee $\Det Df=J_f$, since the function $P_\beta$ satisfies the assumptions (i) and (ii) of Theorem~1.2 in~\cite{koszho}. The desired equality $\Det Df=J_f$ follows also from the remark in~\cite[p.~594]{greco2} All this makes the application of Lemma~\ref{lemma} possible, concluding the proof of the theorem.
\end{proof}
\begin{lemma}\label{lemma}
Let $f\in W^{1,q}_\loc(\Omega;\bR^n)$, $\Omega\subset\bR^n$ ($n\geq2$ and $p>n-1$), be a homeomorphism, such that $\Det Df=J_f$, $J_f(x)\geq0$ for almost every $x\in\Omega$ and $J_f\log^\beta\log(e^e+J_f)\in L^1_\loc$ for some $\beta$. If $n>2$, assume in addition that $f^{-1}\in W^{1,q}_\loc(\Omega,\bR^n)$ for some $q\in]n-1,n[$. Then $\mathcal H^{h_{n,\beta}}(f(E))=0$, whenever $E\subset\Omega$ is such that $\dim_{\mathcal H}E<n$.
\end{lemma}
The assumptions $f\in W^{1,q}_\loc(\Omega;\bR^n)$ and $\Det Df=J_f$ are due to our intention to use Lemma~3.2 from~\cite{KKM}. Before proving Lemma~\ref{lemma}, let us state the following auxillary result from~\cite[Lemma~4]{ndim} (see~\cite[Lemma~3.1]{KZZ1} for the planar case).
\begin{lemma}\label{decom}$ $
\begin{itemize}
\item[(\emph{i})]
Let $f\colon\Omega\to f(\Omega)\subset\bR^n$, $n>2$, be a homeomorphism such that $f^{-1}\in
W^{1,q}_\loc(\Omega,\bR^n)$ for some $q\in]n-1,n[$. Then there exists a set $F\subset
f(\Omega)$ such that $\mathcal{H}^{n-\frac{q}{2}}(F)=0$ and for all $y\in
f(\Omega)\setminus F$ there exist constants $C_y>0$ and $r_y>0$
such that 
\begin{equation}\label{ineq:desired}
\diam(f^{-1}(B(y,r)))\leq C_yr^{1/2},
\end{equation}
for all $0<r<r_y$.

\item[(\emph{ii})] If $n=2$, (i) is true with the assumption $f^{-1}\in W^{1,q}_\loc(\Omega,\bR^n)$ replaced by the condition $f\in W^{1,1}_\loc(\Omega)$ and with $q=1$, that is, with $\mathcal{H}^{3/2}(F)=0$ for the exceptional set $F$.
\end{itemize}
\end{lemma}
\begin{proof}[Proof of Lemma~\ref{lemma}]
The proof repeats the strategy of the proof of Theorem~1.1 from~\cite{tapio}. As in Lemma~3.2 from~\cite{KZZ1}, using Lemma~\ref{decom}, we may represent the image set $\Omega^\prime=f(\Omega)$ in the following form
$$
\Omega^\prime=F\cup\bigcup\limits_{j=1}^\infty\bigcup\limits_{k=1}^\infty\bigl\{y\in
\Omega^\prime\,|\,\diam(f^{-1}(B(y,r)))\leq kr^{\frac{1}{2}}\text{ for all
}r\in]0,1/j[\bigr\},
$$
obtaining a decomposition $\Omega^\prime= \bigcup_{i=0}^\infty F_i$ and a collection of constants $\{C_i\}_{i=1}^\infty$, $\{R_i\}_{i=1}^\infty$, such that $\mathcal H^{h_{n,\beta}}(F_0)=0$ and for each $i=1,2, \dots$, we have $1\leq C_i < \infty$, $R_i>0$ and
\begin{equation}\label{star}
 f^{-1}\left((f(A)\cap F_i)+\left(\frac{r}{C_i}\right)^2\right) \subset A + r
\end{equation}
for every $A \subset \Omega$ and for every $r \in ]0,R_i[$.

Fix $i\geq1$. Let us show that $\mathcal H^{h_{n,\beta}}(f(E)\cap F_i)=0$. Take some $$s\in]\max\{\dim_{\mathcal H}E,n-1\}, n[$$ and put $\sigma=\frac{n-s}{2}<\frac{1}{2}$. Choose $r_0\in]0,e^{-1/{\sigma^2}}[$ small enough to guarantee $\log^\beta(2\log\frac{C_i}{r_0})\leq r_0^{-\sigma}$.

Fix now $\varepsilon>0$. Using the absolute continuity of the Lebesgue integral and the given integrability of the Jacobian, we may find a number $\delta>0$, such that
$$
\int_A J_f(x)\log^\beta\log(e^e+J_f(x))dx<\varepsilon
$$
for each $A\subset\Omega$ such that $\mathcal L^n(A)<\delta$.

Since $\mathcal H^s(E)=0$, we may find a countable collection of balls\linebreak $\{B(x_j,r_j)\}_{j=1}^\infty$, covering $E$ and having radii less than $\min\{r_0, R_i, \frac{1}{C_i}\}$, such that
$$
\sum\limits_{j=1}^\infty 2^n\omega_n r_j^s <\min\{\varepsilon,\delta\}.
$$

Now, write $F_{i,j} = F_i \cap f(B(x_j,r_j))$ for each $j\in\bN$. Notice by~\eqref{star} that $f^{-1}(F_{i,j} + R_{i,j}) \subset B(x_j,2r_j)$, where $R_{i,j} = (\frac{r_j}{C_i})^2$.

Next, we use the $5r$-covering theorem to find an at most countable subcollection of pairwise disjoint balls $\{B(y_k,\rho_k)\}_{k\in K}$ from the collection
$$
\bigcup\limits_{j=1}^\infty\{B(y,R_{i,j}) : y \in F_{i,j}\}
$$
so that
$$
 F_i \cap f(E) \subset \bigcup_{k\in K} B(y_k,5\rho_k),
$$
where, for each $k\in K$, we have $y_k\in F_{i,j}$ for some $j=j(k)$ and $\rho_k=R_{i,j(k)}$.

Since $r_j<e^{-1/{\sigma^2}}<e^{-4}$ for all $j\in\bN$, we have $\frac{1}{10R_{i,j(k)}}>\frac{C_i^2 e^8}{10}>e$ for $k\in K$. Lemma~3.2 from~\cite{KKM} yields
$$
\mathcal{L}^n(B(y_k,R_{i,j(k)}))\leq\int_{f^{-1}(B(y_k,R_{i,j(k)}))}J_f(x)dx
$$
for all $k\in K$. Thus, we may estimate
\begin{align*}
\mathcal{H}&_{10r_0}^{h_{n,\beta}}(F_i \cap f(E)) \leq \sum_{k\in K}
10^n R_{i,j(k)}^n
\log^\beta\log\Bigl(\frac{1}{10R_{i,j(k)}}\Bigr) \\
\leq & \frac{10^n}{\omega_n}\sum_{k\in K}
\mathcal{L}^n(B(y_k,R_{i,j(k)}))\log^\beta\log\Bigl(\frac{1}{R_{i,j(k)}}\Bigr)\\
\leq &\frac{10^n}{\omega_n}
\sum_{k\in K}\int_{f^{-1}(B(y_k,R_{i,j(k)}))}\log^\beta\log\Bigl(\frac{1}{R_{i,j(k)}}
\Bigr)J_f(x)dx\\
= & \frac{10^n}{\omega_n}\sum_{k\in K}\bigg(\int_{\{x \in f^{-1}(B(y_k,R_{i,j(k)})): J_f(x) <
r_{j(k)}^{-\sigma}\}}\log^\beta\log\Bigl(\frac{1}{R_{i,j(k)}}\Bigr)J_f(x)dx\\
 &+ \int_{\{x \in f^{-1}(B(y_k,R_{i,j(k)})): J_f(x) \geq
r_{j(k)}^{-\sigma}\}}\log^\beta\log\Bigl(\frac{1}{R_{i,j(k)}}\Bigr)J_f(x)dx\bigg)\\
\leq& \frac{10^n}{\omega_n}\sum_{k\in K} r_{j(k)}^{-2\sigma}\mathcal{L}^n(f^{-1}(B(y_k,R_{i,j(k)})))\\
 &+\frac{10^n}{\omega_n}\sum_{k\in K}\frac{\log^\beta\log(1/R_{i,j(k)})}{\log^\beta\log(e^e +
1/r_{j(k)}^\sigma)}
\int_{f^{-1}(B(y_k,R_{i,j(k)}))}J_f\log^\beta\log(e^e + J_f),
\end{align*}
using the fact that $\log^\beta(2\log\frac{C_i}{r_j})\leq r_j^{-\sigma}$ for all $j\in\bN$. Let us estimate the first
term in the last sum. By grouping the balls according to $j(k)$ and using the relation $f^{-1}(F_{i,j} + R_{i,j}) \subset B(x_j,2r_j)$, we get
\begin{multline*}
 \sum_{k\in K} r_{j(k)}^{-2\sigma}\mathcal{L}^n(f^{-1}(B(y_k,R_{i,j(k)}))) =
\sum_{j=1}^\infty r_j^{s-n}\sum_{\substack{k\in K\\j(k)=j}}
\mathcal{L}^n(f^{-1}(B(y_k,R_{i,j})))\\
\leq\sum_{j=1}^\infty r_j^{s-n}\mathcal{L}^n(B(x_j,2r_j)) = \sum_{j=1}^\infty
2^n\omega_n r_j^s < \varepsilon.
\end{multline*}
Let us now estimate the second term in the sum. Since $r_j<\frac{1}{C_i}$ and $r_j<e^{-1/{\sigma^2}}<e^{-4}$ for all $j\in\bN$, we obtain for each $k\in K$
\begin{align*}
\frac{\log^\beta\log(1/R_{i,j(k)})}{\log^\beta\log(e^e+1/r_{j(k)}^\sigma)}\leq&
\frac{\log^\beta\bigl(2\log\frac{C_i}{r_{j(k)}}\bigr)}{\log^\beta\bigl(\sigma\log\frac{1}{r_{j(k)}}\bigr)}\leq
\frac{\log^\beta\bigl(4\log\frac{1}{r_{j(k)}}\bigr)}{\log^\beta\bigl(\sigma\log\frac{1}{r_{j(k)}}\bigr)}\\=&
\left(\frac{\log4+\log\log\frac{1}{r_{j(k)}}}{\log\sigma+\log\log\frac{1}{r_{j(k)}}}\right)^\beta\leq2^{2\beta}.
\end{align*}
Using again the fact that $f^{-1}(F_{i,j} + R_{i,j}) \subset B(x_j,2r_j)$ for all $j\in\bN$, we conclude
\begin{align*}
 \sum_{k\in K}&\frac{\log^\beta\log(1/R_{i,j(k)})}{\log^\beta\log(e^e + 1/r_{j(k)}^\sigma)}
\int_{f^{-1}(B(y_k,R_{i,j(k)}))}J_f\log^\beta\log(e^e + J_f) \\
&\leq 2^{2\beta}
\sum_{k\in K} \int_{f^{-1}(B(y_k,R_{i,j(k)}))}J_f\log^\beta\log(e^e + J_f)\\
&\leq 2^{2\beta}
\int_{\bigcup_{k\in K}f^{-1}(B(y_k,R_{i,j(k)}))}J_f\log^\beta\log(e^e + J_f)\\
&\leq 2^{2\beta}
\int_{\bigcup_{j=1}^\infty B(x_j,2r_j)}J_f\log^\beta\log(e^e + J_f) \leq
2^{2\beta}\varepsilon,
\end{align*}
since
\[
 \mathcal{L}^n\left(\bigcup_{j=1}^\infty B(x_j,2r_j)\right) \leq
\sum_{j=1}^\infty 2^n\omega_n r_j^n \le \sum_{j=1}^\infty 2^n\omega_n r_j^s < \delta.
\]

\end{proof}

\section{Planar case}

As it was mentioned in the first section, the assumption on $f$ to be a homeomorphism can be avoided in the plane due to factorization of the solutions of the Beltrami equation. The \emph{Beltrami equation} is an equation in the complex plane $\mathbb C$ of the form
\begin{equation}\label{eq:beltrami0}
\overline{\partial}f(z)=\mu(z)\partial f(z),
\end{equation}
where $\overline{\partial}=\frac{1}{2}(\partial_x+i\partial_y)$ and $\partial=\frac{1}{2}(\partial_x-i\partial_y)$. The function $\mu$ is the \emph{Beltrami coefficient} of the mapping $f$ (provided $f$ is a solution of \eqref{eq:beltrami0} in some sense). Given an abstract Beltrami coefficient $\mu(z)$, such that $|\mu(z)|<1$ almost everywhere, we can associate to $\mu$ a real-valued function $K=\frac{1+|\mu|}{1-|\mu|}$, called a \emph{distortion function} of the Beltrami equation. The terminology is natural, as the Beltrami equation yields the distortion inequality
$$
|Df(z)|^2\leq K(z)J_f(z)
$$
for its $W^{1,1}_\loc$--solutions. Conversely, a mapping $f$ with finite 
distortion function $K_f(z)$ satisfies almost everywhere the Beltrami equation 
with the associated Beltrami coefficient $\mu_f(z)=\overline{\partial}f(z)/\partial f(z)$. In this case, the distortion function of this Beltrami equation equals $K_f$ and $|\mu(z)|\leq\frac{K(z)-1}{K(z)+1}<1$ for almost every $z$.

\begin{proof}[Proof of Theorem~\ref{planar}]
Let $\mathcal A$ be defined by $\mathcal A(t)=p\frac{t}{1+\log t}-p$. Thus, our sub-exponential integrability assumption on $f$ may be rewritten as $e^{\mathcal A(K_f(z))}\in L^1(\Omega)$. Clearly, the function $\mathcal A$ satisfies conditions~1\nobreakdash--3 from~\cite[pp.~570--571]{book}, so, we may apply Theorem~20.5.2 in~\cite{book}, which gives the unique principal solution $g$ to the global Beltrami equation, satisfied by $f$ almost everywhere in $\Omega$. See~\cite[Definition~20.0.4]{book} for the definition of the principal solution of the Beltrami equation. In particular, $g$ is homeomorphic. In addition, Theorem~20.5.2 in~\cite{book} asserts that $f$ can be factorized as $f=\phi\circ g$ (where $\phi$ is holomorphic in $g(\Omega)$), provided $f\in W^{1,P}_\loc(\Omega)$ for
$$
P(t)=\begin{cases}
      t^2, & 0\leq t\leq 1,\\
      \frac{t^2}{\mathcal A^{-1}(\log t^2)}, & t\geq1,
     \end{cases}
$$
which is true by~\cite[Theorem~20.5.1]{book}.

Higher integrability of the Jacobian for $g$ follows from Theorem~1 in~\cite{gill}, yielding $J_g\log^\beta\log(e^e+J_g)\in L^1_\loc$ for all $\beta<p$. This allows to use Lemma~\ref{lemma}, giving $\mathcal H^{h_{2,\beta}}(g(E))=0$ for all $\beta<p$ and  each set $E\subset\Omega$ such that $\dim_{\mathcal H}E<2$. Finally, as $\phi$ is locally Lipschitz, we obtain $\mathcal H^{h_{2,\beta}}(f(E))=0$ for such $\beta$ and $E$.
\end{proof}

\subsection*{Acknowledgments} The authors thank Pekka Koskela for suggesting this problem.

\bibliographystyle{amsplain}
\bibliography{jyv6}
\end{document}